\documentclass[11pt]{article}

\usepackage{amsmath, amssymb}
\usepackage{tikz}

\usepackage{authblk}
\usepackage{geometry}
\usepackage{hyperref}
\usepackage{graphicx}
\usepackage{epstopdf}
\usepackage{mathrsfs}
\usepackage{amsfonts}
\usepackage{amscd,amsthm,bm,psfrag}
\usepackage[numbers,sort&compress]{natbib}

\makeatletter
\renewcommand{\@seccntformat}[1]{{\csname the#1\endcsname}{\normalsize.}\hspace{.5em}}

\makeatother

\numberwithin{equation}{section}

\def \[{\begin{equation*}}
\def \]{\end{equation*}}

\newtheorem{thm}{Theorem}[section]

\newtheorem{lem}[thm]{Lemma}

\newtheorem{prop}[thm]{Proposition}
\newtheorem{rem}[thm]{Remark}

\newtheorem{prob}[thm]{Problem}

\newtheorem*{thm*}{Theorem}
\newtheorem*{prop*}{Proposition}

\newcommand\cK{{\mathcal K}}

\newcommand\cN{{\mathcal N}}

\newcommand\ex{\ensuremath{\mathrm{ex}}}
\newcommand\tw{\ensuremath{\operatorname{tw}}}

\def\final{0}  % set this to 1 to get a comment-free version
\def\iflong{\iffalse}
\ifnum\final=0  %namely if we allow comments in the output
\newcommand{\jnote}[1]{{\color{blue}[{\tiny \textbf{Junpeng:} \bf #1}]\marginpar{\color{blue}*}}}
\newcommand{\ynote}[1]{{\color{purple}[{\tiny \textbf{Yuhang:} \bf #1}]\marginpar{\color{purple}*}}}
\else % in this case [final=1] we don't want any comments to show
\newcommand{\jnote}[1]{}
\newcommand{\ynote}[1]{}
\fi

\begin{document}

\title{Counterexamples to a treewidth conjecture on generalized Tur\'an problems}

\author[a,b,*]{Junpeng Zhou\,$^{\rm a,b}$, \ Xiying Yuan\,}

\affil[a]{\small \,Department of Mathematics, Shanghai University,
Shanghai 200444, PR China}
\affil[b]{\small \,Newtouch Center for Mathematics of Shanghai University,
Shanghai 200444, PR China}

\date{}

\maketitle

\footnotetext{*\textit{Corresponding author}.}
\footnotetext{
Email addresses:
\texttt{junpengzhou@shu.edu.cn} (J.~Zhou),
\texttt{xiyingyuan@shu.edu.cn} (X.~Yuan).}

\begin{abstract}
Given graphs $H$ and $F$, the generalized Tur\'{a}n number ${\rm ex}(n,H,F)$ is the maximum number of copies of $H$ in an $n$-vertex $F$-free graph. Alon and Shikhelman (J. Combin. Theory Ser. B, 2016) initiated the systematic study of generalized Tur\'{a}n problems. Recently, Gao, Wu and Xue (J. Graph Theory, 2026) asked whether every graph $F$ with chromatic number $\chi(F)=r\geq3$ and treewidth $\tw(F)\geq r$ satisfies $\ex(n,K_r,F)=\Omega(n^{r-1})$.
In this note, we give a negative answer to this question for every $r\geq3$. More precisely, we prove that the graph $F_r=K_{r-3}\vee H$, where $H$ is obtained from $K_4$ by subdividing one edge once, satisfies $\chi(F_r)=\tw(F_r)=r$ and
\[
n^{r-1}e^{-O(\sqrt{\log n})}\leq \ex(n,K_r,F_r)=o(n^{r-1}).
\]
This result also disproves Conjecture~6.3 in the recent survey of Gerbner and Palmer (Electron. J. Combin., 2026).
\end{abstract}


{\noindent{\bf AMS (2020) subject classifications:} 05C35}

\section{\normalsize Introduction}\label{sec:intro}
Throughout this paper, all graphs are finite, simple and undirected. Given graphs $T$ and $F$, a graph $G$ is called \textit{$F$-free} if it does not contain a copy of $F$ as a subgraph. Let $\cN(T,G)$ denote the number of copies of $T$ in $G$. Let
$$
 \ex(n,T,F)=\max\bigl\{\cN(T,G): |V(G)|=n \text{ and } G \text{ is } F\text{-free}\bigr\}.
$$
Alon and Shikhelman \cite{AlSh} initiated the systematic study of the function $\ex(n,T,F)$, which is often called the \textit{generalized Tur\'{a}n problem}. When $T=K_2$, the generalized Tur\'{a}n number reduces to the classical Tur\'{a}n number, that is, $\ex(n,K_2,F)=\ex(n,F)$.
For a recent survey on generalized Tur\'{a}n problems, one can refer to the work of Gerbner and Palmer \cite{GePa}.

A graph is \textit{chordal} if it contains no induced cycle of length at least four. The \textit{treewidth} $\operatorname{tw}(G)$ of a graph $G$ is one less than the minimum clique number among all chordal graphs containing $G$ as a subgraph. %The \textit{treewidth} $\operatorname{tw}(G)$ of a graph $G$ is one less than the size of the largest clique in a chordal graph containing $G$ with the smallest clique number.
Recently, Gao, Wu and Xue \cite{GaWuXu} showed that if a graph $F$ satisfies $\operatorname{tw}(F)\leq r-1$, then $\operatorname{ex}(n,K_r,F)=O(n^{r-1})$. They also posed the following question.

\begin{prob}[Gao, Wu and Xue \cite{GaWuXu}]\label{prob1}
Is it true that if $\chi(F)=r\geq3$ and $\operatorname{tw}(F)\geq r$, then $\operatorname{ex}(n,K_r,F)=\Omega(n^{r-1})$?
\end{prob}

Gerbner and Palmer~\cite{GePa} subsequently reformulated Problem~\ref{prob1} as Conjecture~6.3 in their survey.

%\begin{con}[Gerbner and Palmer \cite{GePa}]\label{con1}
%If $F$ is a graph with $\chi(F)=r\geq3$ and $\operatorname{tw}(F)\geq r$, then
%\[
%\operatorname{ex}(n,K_r,F)=\Omega(n^{r-1}).
%\]
%\end{con}

In this note, we give counterexamples for every $r\geq3$. Let $H$ be the graph with vertex set $V(H)=\{1,2,3,4,5\}$ and edge set $E(H)=\{12,13,14,23,24,35,45\}$. Namely, $H$ is obtained from the copy of $K_4$ on $\{1,2,3,4\}$ by replacing the edge $34$ with the path $354$. For $r\geq3$, let $F_r=K_{r-3}\vee H$, where $\vee$ denotes the join of two vertex-disjoint graphs and $K_0$ denotes the empty graph (see Figure~\ref{fig:Fr}). Clearly, $|V(F_r)|=r+2$.

\begin{figure}[ht]
\centering
\begin{tikzpicture}[
 vertex/.style={
  circle,
  minimum size=3.6mm,
  inner sep=0pt
 },
 redvertex/.style={
  vertex,
  draw=red!75!black,
  fill=red!75!black
 },
 bluevertex/.style={
  vertex,
  draw=blue!75!black,
  fill=blue!75!black
 },
 greenvertex/.style={
  vertex,
  draw=green!55!black,
  fill=green!55!black
 },
 hedge/.style={
  draw=black!85,
  line width=0.9pt
 },
 joinedge/.style={
  draw=gray!55,
  line width=0.55pt
 }
]

% The clique K_{r-3}
\node[
 circle,
 draw=orange!80!black,
 fill=orange!12,
 minimum size=20mm,
 font=\large
] (K) at (-2.75,0) {$K_{r-3}$};

% Coordinates of H
\coordinate (v4) at (0,1.15);
\coordinate (v1) at (2.25,1.15);
\coordinate (v3) at (0,-1.15);
\coordinate (v2) at (2.25,-1.15);
\coordinate (v5) at (0,0);

% The complete join between K_{r-3} and H
\foreach \i in {1,2,3,4,5}{
 \draw[joinedge] (K.east) -- (v\i);
}

% Edges of H
\draw[hedge] (v4)--(v1)--(v2)--(v3);
\draw[hedge] (v1)--(v3);
\draw[hedge] (v4)--(v2);
\draw[hedge] (v4)--(v5)--(v3);

% A proper 3-coloring of H
\node[redvertex] (p1) at (v1) {};
\node[bluevertex] (p2) at (v2) {};
\node[greenvertex] (p3) at (v3) {};
\node[greenvertex] (p4) at (v4) {};
\node[redvertex] (p5) at (v5) {};

% Vertex labels
\node[inner sep=1pt,above right=0pt] at (p1.north east) {$1$};
\node[inner sep=1pt,below right=0pt] at (p2.south east) {$2$};
\node[inner sep=1pt,below left=0pt] at (p3.south west) {$3$};
\node[inner sep=1pt,above left=0pt] at (p4.north west) {$4$};
\node[inner sep=1pt,right=0pt] at (p5.east) {$5$};

\node at (1.12,-1.58) {$H$};

\end{tikzpicture}
\caption{{\small The graph $F_r$}}
\label{fig:Fr}
\end{figure}
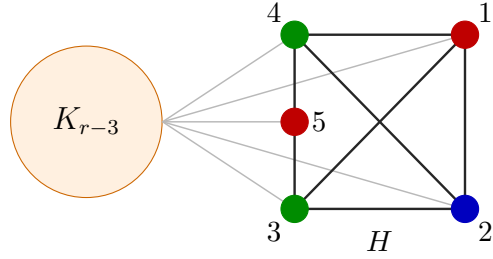

Our main result is the following.

\begin{thm}\label{thm:main}
For every integer $r\geq3$, the graph $F_r$ satisfies $\chi(F_r)=\tw(F_r)=r$. Moreover,
\[
n^{r-1}e^{-O(\sqrt{\log n})}\leq \ex(n,K_r,F_r)=o(n^{r-1}).
\]
\end{thm}

Theorem~\ref{thm:main} gives a negative answer to Problem~\ref{prob1} for every $r\geq3$, and thereby disproves Conjecture~6.3 in the recent survey of Gerbner and Palmer~\cite{GePa}.

\begin{rem}
Note that the graph $F_r$ has the minimum possible order among graphs satisfying the assumptions of Problem~\ref{prob1}. Indeed, let $F$ be an $r$-chromatic graph with at most $r+1$ vertices. If $|V(F)|=r$, then $F=K_r$ and $\tw(F)=r-1$. If $|V(F)|=r+1$, then $F\neq K_{r+1}$, and thus $F$ is a subgraph of $K_{r+1}-e$ for some edge $e$. Since $K_{r+1}-e$ is chordal with clique number $r$, we have $\tw(F)\leq r-1$ in either case.
\end{rem}

\begin{rem}
The lower bound in Theorem~\ref{thm:main} shows that there exists a constant $c_r>0$ such that
\[
\ex(n,K_r,F_r)\geq n^{r-1}e^{-c_r\sqrt{\log n}}
\]
for all sufficiently large $n$. It follows that for every fixed $\varepsilon>0$,
\[
\frac{n^{r-1}e^{-c_r\sqrt{\log n}}}{n^{r-1-\varepsilon}}=e^{\varepsilon\log n-c_r\sqrt{\log n}}\longrightarrow\infty.
\]
This implies that although $\ex(n,K_r,F_r)=o(n^{r-1})$, no upper bound of the form $O(n^{r-1-\varepsilon})$ holds for any fixed $\varepsilon>0$.
\end{rem}

The rest of the paper is organized as follows. In Section~\ref{sec:triangles}, we prove a triangle-counting lemma for $H$-free graphs that will be used in the proof of Theorem~\ref{thm:main}. In Section~\ref{sec:proof}, we prove Theorem~\ref{thm:main}. Finally, we conclude with several remarks in Section \ref{sec:remarks}.

\section{\normalsize Counting triangles in \texorpdfstring{$H$}{H}-free graphs}\label{sec:triangles}
For a graph $G$ and a vertex $x\in V(G)$, let $N_G(x)$ denote the \textit{neighborhood} of $x$ in $G$. For an edge $e=xy$ of $G$, let $N_G(e):=N_G(x)\cap N_G(y)$ and $d_G(e):=|N_G(e)|$. Note that $d_G(e)$ is exactly the number of triangles in $G$ containing the edge $e$.

Recall that the triangle removal lemma of Ruzsa and Szemer\'edi~\cite{RuSz} can be stated as follows (see also~\cite[Theorem~1.1]{GiShWi}).

\begin{lem}[Ruzsa and Szemer\'edi~\cite{RuSz}]\label{lem:removal}
If an $n$-vertex graph contains $o(n^3)$ triangles, then it can be made triangle-free by deleting $o(n^2)$ edges.
\end{lem}

We now establish the following bound on the number of triangles in $H$-free graphs.

\begin{lem}\label{lem:triangle}
For every $\varepsilon>0$, there exists an integer $n_0$ such that
\[
\ex(n,K_3,H)<\varepsilon n^2
\]
for every $n\geq n_0$.
\end{lem}

\begin{proof}[\bf Proof]
Let $G$ be an $n$-vertex $H$-free graph. We claim that for any two distinct edges $e,f\in E(G)$,
\begin{equation*}
|N_G(e)\cap N_G(f)|\leq1.
\end{equation*}
Label the endpoints of $e$ by $1$ and $2$. Suppose to the contrary that $N_G(e)\cap N_G(f)$ contains two distinct vertices, say $3$ and $4$. Since $f\neq e$, one endpoint of $f$ (say $5$) lies outside $\{1,2\}$. Moreover, $5\notin \{3,4\}$, since neither endpoint of $f$ belongs to $N_G(f)$. Hence, the seven edges $12,13,14,23,24,35,45$ form a copy of $H$, which contradicts that $G$ is $H$-free.

By the above claim, the families $\binom{N_G(e)}{2}$ for $e\in E(G)$ are pairwise disjoint. Therefore,
\begin{equation*}
\sum_{e\in E(G)}\binom{d_G(e)}2\leq \binom n2.
\end{equation*}

For an integer $N\geq2$, let
\[
 E_N=\{e\in E(G): d_G(e)\geq N\}.
\]
The number of triangles containing at least one edge of $E_N$ is at most
\begin{align}\label{eq2.1}
\sum_{e\in E_N}d_G(e)&\leq \sum_{e\in E_N}\frac{d_G(e)(d_G(e)-1)}{N-1}= \frac{2}{N-1}\sum_{e\in E_N}\binom{d_G(e)}2 \notag\\
&\leq \frac{2}{N-1}\binom{n}{2}= \frac{n(n-1)}{N-1}.
\end{align}

Let $G_N:=G-E_N$. Then every edge $e\in E(G_N)$ satisfies $d_{G_N}(e)\leq d_G(e)<N$. We claim that for every fixed $N$,
\begin{equation}\label{eq2.2}
\cN(K_3,G_N)=o(n^2).
\end{equation}
Suppose otherwise. Then there exist $\varepsilon>0$ and a sequence of such graphs satisfying $\cN(K_3,G_N)\geq\varepsilon n^2$. Since every edge of $G_N$ belongs to fewer than $N$ triangles, a greedy procedure yields at least $\lfloor\varepsilon n^2/(3N)\rfloor$ edge-disjoint triangles. On the other hand,
\[
3\cN(K_3,G_N)= \sum_{e\in E(G_N)}d_{G_N}(e)\leq |E(G_N)|\cdot(N-1)=O(n^2)=o(n^3).
\]
By Lemma \ref{lem:removal}, there is a set of $o(n^2)$ edges whose deletion makes $G_N$ triangle-free. However, such a set must contain at least one edge from each of the $\lfloor\varepsilon n^2/(3N)\rfloor$ edge-disjoint triangles, which is a contradiction. This proves~\eqref{eq2.2}.

Combining~\eqref{eq2.1} and~\eqref{eq2.2}, we obtain that for every fixed $N\geq2$,
\[
\cN(K_3,G)\leq \frac{n(n-1)}{N-1}+o(n^2).
\]
Thus, $\ex(n,K_3,H)\leq \frac{n(n-1)}{N-1}+o(n^2)$. Let $\varepsilon>0$ and fix an integer $N\geq2$ such that $1/(N-1)<\varepsilon/2$. For this fixed $N$, the second term is at most $\varepsilon n^2/2$ for all sufficiently large $n$. Hence, $\ex(n,K_3,H)<\varepsilon n^2$ for all sufficiently large $n$. Since $\varepsilon>0$ is arbitrary, the result follows.
\end{proof}

\section{\normalsize Proof of Theorem~\ref{thm:main}}\label{sec:proof}
For a graph $G$ and a vertex set $U\subseteq V(G)$, let $G[U]$ be the subgraph of $G$ induced by $U$. First, we prove the structural properties of $F_r$.

\begin{prop}\label{prop:parameters}
For every $r\geq3$, the graph $F_r$ satisfies
$\chi(F_r)=\tw(F_r)=r$.
\end{prop}
\begin{proof}[\bf Proof]
Note that the triangles $123$ and $124$ show that $\chi(H)\geq3$. Conversely, assign distinct colors to $1,2,3$, and give $4$ and $5$ the colors of $3$ and $1$, respectively. This is a proper $3$-coloring of $H$, so $\chi(H)=3$. Thus, by the definition of $F_r$,
\[
 \chi(F_r)=\chi(K_{r-3})+\chi(H)=r.
\]

It remains to determine the treewidth of $F_r$. Let $H'$ be obtained from $H$ by adding the edge $34$. Then $H'$ is a chordal graph whose maximal cliques are $\{1,2,3,4\}$ and $\{3,4,5\}$. Hence, $\tw(H)\leq3$. Conversely, let $J$ be any chordal graph containing $H$. The cycles $13541$ and $23542$ are induced $4$-cycles in $H$. Since $J$ is chordal, it contains either the edge $15$ or the edge $34$, and either the edge $25$ or the edge $34$. If $34\in E(J)$, then $\{1,2,3,4\}$ is a clique in $J$. Otherwise, both $15$ and $25$ belong to $E(J)$, and hence $\{1,2,3,5\}$ is a clique in $J$. Thus, every chordal graph containing $H$ has clique number at least $4$, which gives $\tw(H)\geq3$. Therefore, $\tw(H)=3$.

Let $X$ be the vertex set of the copy of $K_{r-3}$ in $F_r$. Observe that the graph $K_{r-3}\vee H'$ is chordal and has clique number $r+1$. Hence, $\tw(F_r)\leq r$. Conversely, if $J$ is any chordal graph containing $F_r$, then $J[V(H)]$ is a chordal graph containing $H$, and thus contains a clique of size $4$. Together with $X$, this clique forms a clique of size $r+1$ in $J$. Thus, $\tw(F_r)\geq r$. This completes the proof.
\end{proof}

\subsection{\normalsize The upper bound}
First, we establish the upper bound on the number of copies of $K_r$ in $F_r$-free graphs.

\begin{prop}\label{prop:upper}
For every integer $r\geq3$,
\[
 \ex(n,K_r,F_r)=o(n^{r-1}).
\]
\end{prop}
\begin{proof}[\bf Proof]
For convenience, set $s:=r-3$. Let $G$ be an $n$-vertex $F_r$-free graph. For each $s$-clique $S$ of $G$, let
\[
 G_S=G\left[\bigcap_{z\in V(S)}N_G(z)\right].
\]
When $s=0$, we take $S=\emptyset$ and $G_S=G$. The graph $G_S$ is $H$-free. Indeed, if $G_S$ contained a copy of $H$, then this copy together with $S$ would form a copy of $F_r$ in $G$, which is a contradiction.

Let $\cK_s(G)$ be the family of $s$-cliques of $G$. By double-counting the pairs $(S,R)$, where $S\in\cK_s(G)$ and $R$ is a copy of $K_r$ in $G$ containing $S$, we obtain
\begin{equation}\label{eq:lift-identity}
 \binom rs\cN(K_r,G)
 =\sum_{S\in\cK_s(G)}\cN(K_3,G_S).
\end{equation}
Indeed, for each $S\in\cK_s(G)$, the copies of $K_r$ containing $S$ are in one-to-one correspondence with the triangles in $G_S$.

Let $\varepsilon>0$. By Lemma~\ref{lem:triangle}, there exists an integer $m_0$ such that every $m$-vertex $H$-free graph with $m\geq m_0$ contains at most $\varepsilon m^2$ triangles. If $|V(G_S)|\geq m_0$, then $\cN(K_3,G_S)\leq\varepsilon n^2$. If $|V(G_S)|<m_0$, then $\cN(K_3,G_S)<\binom{m_0}{3}$. Hence, for every $S\in\cK_s(G)$,
\[
 \cN(K_3,G_S)\leq\varepsilon n^2+\binom{m_0}{3}.
\]
It follows from~\eqref{eq:lift-identity} and $|\cK_s(G)|\leq\binom ns$ that
\[
 \cN(K_r,G)\leq \frac{\binom ns}{\binom rs}\left(\varepsilon n^2+\binom{m_0}{3}\right)= \frac{\binom {n}{r-3}}{\binom {r}{r-3}}\left(\varepsilon n^2+\binom{m_0}{3}\right).
\]
Thus,
\begin{equation*}
\ex(n,K_r,F_r)\leq \frac{\binom {n}{r-3}}{\binom {r}{r-3}}\left(\varepsilon n^2+\binom{m_0}{3}\right)=\frac{\varepsilon n^{r-1}}{(r-3)!\binom{r}{r-3}}+o(n^{r-1}).
\end{equation*}
Since $\varepsilon>0$ can be chosen arbitrarily small, the proposition follows. %Since $\varepsilon>0$ is arbitrary, the proposition follows.
\end{proof}

\subsection{\normalsize The lower bound}
Let $B_{r,r-1}$ be the graph formed by two copies of $K_r$ sharing exactly $r-1$ vertices. Gowers and Janzer~\cite{GoJa} proved that for every integer $r\geq3$, there is an $n$-vertex graph $G$ with
\[
n^{r-1} e^{-O(\sqrt{\log n})}=n^{r-1-o(1)}
\]
copies of $K_r$ such that every copy of $K_{r-1}$ in $G$ is contained in at most one copy of $K_r$. In particular, $G$ is $B_{r,r-1}$-free. Otherwise, some copy of $K_{r-1}$ would be contained in two distinct copies of $K_r$.

\begin{proof}[\bf Proof of Theorem~\ref{thm:main}]
Note that the two triangles $123$ and $124$ in $H$ share exactly the edge $12$. Together with the copy of $K_{r-3}$, they form two copies of $K_r$ in $F_r$ sharing exactly $r-1$ vertices. Namely, $F_r$ contains a copy of $B_{r,r-1}$. Therefore, by the above result of Gowers and Janzer~\cite{GoJa}, we obtain
\[
 \ex(n,K_r,F_r)\geq  \ex(n,K_r,B_{r,r-1})\geq n^{r-1} e^{-O(\sqrt{\log n})}.
\]
The upper bound follows from Proposition~\ref{prop:upper}, and Proposition~\ref{prop:parameters} establishes that $\chi(F_r)=\tw(F_r)=r$. This completes the proof.
\end{proof}

\section{\normalsize Concluding remarks}\label{sec:remarks}
In this note, we construct counterexamples that give a negative answer to Problem~\ref{prob1} posed by Gao, Wu and Xue \cite{GaWuXu}, and disprove Conjecture~6.3 of Gerbner and Palmer \cite{GePa}. In particular, these counterexamples satisfy $\tw(F_r)=\chi(F_r)=r$. However, they do not settle whether the conclusion of Problem~\ref{prob1} holds under the stronger assumption $\tw(F)>\chi(F)$. Therefore, it is natural to pose the following strengthened version of Problem~\ref{prob1}.

\begin{prob}\label{prob:strict}
Is it true that if $\chi(F)=r\geq3$ and $\tw(F)>r$, then $\ex(n,K_r,F)=\Omega(n^{r-1})$?
\end{prob}

For an $r$-chromatic graph $F$, let $\sigma(F)$ denote the minimum size of a color class among all proper $r$-colorings of $F$. Observe that the lower bound in Problem~\ref{prob:strict} holds whenever $\sigma(F)\geq2$, even without any assumption on the treewidth of $F$. Indeed, let $G$ be an $n$-vertex complete $r$-partite graph with one part of size $\sigma(F)-1$ and the remaining $r-1$ parts as equal as possible. Clearly, $G$ is $F$-free. %We claim that $G$ is $F$-free. Otherwise, any copy of $F$ in $G$ would induce a proper $r$-coloring of $F$ with a color class of size at most $\sigma(F)-1$, contradicting the definition of $\sigma(F)$.
Moreover, $\cN(K_r,G)=\Omega(n^{r-1})$. Thus, if a counterexample to Problem~\ref{prob:strict} exists, it must satisfy $\sigma(F)=1$.
Let $\{v\}$ be a singleton color class in a proper $r$-coloring of $F$. Then $\chi(F-v)\leq r-1$. Since $\chi(F)\leq\chi(F-v)+1$, we have $\chi(F-v)=r-1$. Furthermore, if $\tw(F)>r$, then deleting a vertex decreases the treewidth by at most one, and hence $\tw(F-v)\geq\tw(F)-1\geq r$. Thus, the remaining case of Problem~\ref{prob:strict} consists of graphs $F$ having a vertex $v$ such that $\chi(F-v)=r-1$ and $\tw(F-v)\geq r$.

\section*{\normalsize Funding}
The research of Zhou and Yuan was supported by the National Natural Science Foundation of China (Nos.~12271337 and 12371347).

\section*{\normalsize Declaration of interest}
The authors declare no known conflicts of interest.

\end{document}